\documentclass[12pt]{amsart}
\usepackage{a4wide}
\usepackage[T1]{fontenc}
\usepackage{amssymb,amsmath,amsthm,latexsym}
\usepackage{mathrsfs}
\usepackage[usenames,dvipsnames]{color}
\usepackage{euscript}
\usepackage{graphicx}
\usepackage{mdwlist}
\usepackage{mathtools,dsfont,wasysym}

\usepackage{stmaryrd}
\usepackage[dvipsnames]{xcolor}

\DeclareFontFamily{U}{mathx}{\hyphenchar\font45}
\DeclareFontShape{U}{mathx}{m}{n}{
      <5> <6> <7> <8> <9> <10>
      <10.95> <12> <14.4> <17.28> <20.74> <24.88>
      mathx10
      }{}

\usepackage{todonotes}

\newtheorem{theorem}{Theorem}[section]
\newtheorem*{theoremA}{Theorem A}
\newtheorem*{theoremB}{Theorem B}

\newtheorem{corollary}[theorem]{Corollary}
\newtheorem{proposition}[theorem]{Proposition}

\theoremstyle{remark}

\newtheorem*{claim*}{Claim}
\theoremstyle{definition}
\newtheorem{definition}[theorem]{Definition}

\numberwithin{equation}{section}
\makeatother

\newcommand{\nn}[1]{{\left\vert\kern-0.25ex\left\vert\kern-0.25ex\left\vert #1 
\right\vert\kern-0.25ex\right\vert\kern-0.25ex\right\vert}}

\renewcommand{\leq}{\leqslant}
\renewcommand{\geq}{\geqslant}

\newcounter{smallromans}

{\end{list}}

\newcommand{\R}{\mathbb{R}}
\newcommand{\N}{\mathbb{N}}

\newcommand{\e}{\varepsilon}

\newcommand{\n}{\left\Vert\cdot\right\Vert}

\newcounter{smallromansdash}

{\end{list}}

\newcounter{bigromans} 
{\end{list}}

\begin{document}
\title[A note on symmetric separation in Banach spaces]{A note on symmetric separation in Banach spaces}

\author[T.~Russo]{Tommaso Russo}
\address[T.~Russo]{Department of Mathematics\\Faculty of Electrical Engineering\\
Czech Technical University in Prague\\Technick\'a 2, 166 27 Praha 6\\ Czech Republic}
\email{russotom@fel.cvut.cz}

\thanks{Research of the author was supported by the project International Mobility of Researchers in CTU CZ.02.2.69/0.0/0.0/16$\_$027/0008465 and by Gruppo Nazionale per l'Analisi Matematica, la Probabilit\`a e le loro Applicazioni (GNAMPA) of Istituto Nazionale di Alta Matematica (INdAM), Italy.}

\date{\today}

\begin{abstract} We present some new results on the symmetric Kottman's constant $K^s(X)$ of a Banach space $X$ and its relationship with the Kottman constant. We show that $K^s(X)>1$, for every infinite-dimensional Banach space, thereby solving a problem by J.M.F.~Castillo and P.L.~Papini. We also investigate such constant in the class of Banach spaces admitting $c_0$ spreading models, answering in particular one question from our previous joint paper with P.~H\'ajek and T.~Kania.
\end{abstract}
\maketitle

\section{Introduction}
The study of distances between unit vectors is an important topic in Banach space theory, whose origin can be traced back at least to the classical Riesz' lemma \cite{Riesz} asserting that the unit ball of every infinite-dimensional normed space contains a $1$-separated sequence (namely, a sequence of vectors whose mutual distances are at least $1$). The result was generalised by Kottman \cite{Kottman}, who constructed in every infinite-dimensional normed space a sequence of unit vectors whose mutual distances are strictly greater than $1$. In the same paper, the author also introduced a parameter, nowadays known as \emph{Kottman's constant}
$$K(X):=\sup\left\{\sigma\geq0\colon \exists\, (x_n)_{n=1}^\infty\subseteq B_X\; \sigma\text{-separated} \right\}$$
and he conjectured that $K(X)>1$ for every infinite-dimensional normed space.

Such a conjecture was indeed solved in the positive by the celebrated Elton--Odell theorem \cite{E-O}. Since then, plenty of results are available in the literature that compute or estimate the Kottman constant for various classes of Banach spaces. As a sample of some these results, and source for several additional references, let us refer to \cite{CGP, CaPa, delpech, KrP, MaPa, prus}.\smallskip

Several variations of the above problem have been considered in the literature. Let us mention, among them, the study of equilateral sets \cite{FOSS14, Koszmider, MV equilateral, mv, Terenzi1, Terenzi2}, antipodal sets \cite{GlMe} and symmetrically separated sets \cite{CGP, CaPa, HKR}; this last notion will be the one relevant for our note.
Let us therefore remind that a subset $A$ of a normed space $X$ is \emph{symmetrically $\delta$-separated} when $\|x\pm y\|\geq\delta$ for any distinct elements $x,y\in A$; accordingly, one also introduces the \emph{symmetric Kottman's constant} as follows:
$$K^s(X):=\sup\left\{\sigma\geq0\colon\exists(x_n)_{n=1}^\infty\subset B_X\; \text{symmetrically }\sigma\text{-separated}\right\}.$$
Such constant has been explicitly defined for the first time in \cite{CaPa}, although the notion of symmetric separation could be traced back at least to \cite[Definition 2.1]{NaSa} and it is closely related to James' work on uniformly non-square Banach spaces \cite{james-distortion} (\emph{cf.} \cite[p.~78-79]{CaPa}).\smallskip

In this context the first natural question (stated as Problem 1 in \cite{CaPa}) is the validity of a symmetric counterpart to the Elton--Odell theorem, \emph{i.e.}, is $K^s(X)>1$ for every infinite-dimensional Banach space $X$? It is, for example, an immediate consequence of James' non-distortion theorem \cite{james-distortion} that $K^s(X)=2$, whenever $X$ contains a copy of $c_0$ or $\ell_1$. Moreover, an adaptation of the same argument to $\ell_p$ yields that $K^s(X)\geq2^{1/p}$, whenever $X$ contains a copy of $\ell_p$ ($1\leq p<\infty$), \cite[Theorem 3]{Kottman}. Castillo and Papini \cite{CaPa} also proved that $K^s(X)>1$ whenever $X$ is uniformly non-square, or a $\mathscr{L}_\infty$-space. Let us refer to the recent paper \cite{HKR} for more detailed references to the literature and for stronger results in the same direction; it is proven, for instance, that $K^s(X)>1$ whenever $X$ contains an infinite-dimensional separable dual Banach space, or an unconditional basic sequence.\smallskip

Although these results cover quite a large class of Banach spaces, the problem in its full generality is not settled by any of the above papers. Our first main contribution in this note consists in showing how to derive a positive answer directly from the Elton--Odell theorem. To wit, we prove the following theorem.

\begin{theoremA} Let $X$ be an infinite-dimensional Banach space. Then, for some $\e>0$, the unit ball of $X$ contains a symmetrically $(1+\e)$-separated sequence.
\end{theoremA}
It is important to to observe that, although Theorem A subsumes many results from \cite{CGP, CaPa, HKR}, it doesn't reduce their interest and, in a sense, motivates them; indeed, the validity of Theorem A tells that the symmetric Kottman's constant is a reasonable variation over Kottman's constant and therefore stimulates its investigation. The proof of the Theorem, together with some quantitative improvements, will be presented in \S \ref{Sec: Ks and K}; a second, similar, its proof will be given in \S \ref{Sec: 2nd proof}.\smallskip

In the second part of our note, \S\ref{Sec: c0 spread}, we shall be concerned with Banach spaces that admit $c_0$ spreading models. In order to motivate such investigation, let us remind that $K^s(X)=2$, whenever $X$ admit a spreading model isomorphic to $\ell_1$, \cite[Corollary 5.6]{HKR}. Since this result heavily depends on a version of James' non-distortion theorem for spreading models, which is also available for Banach spaces with $c_0$ spreading models (compare \cite[Proposition II.2.4]{BeaLa} with \cite[Lemma III.2.4]{BeaLa}), it was temptful to conjecture that $K^s(X)=2$ when $X$ admits a $c_0$ spreading model. On the other hand, the prototypical example of a symmetrically $2$-separated sequence in $c_0$ is $-e_{n+1}+\sum_{j=1}^n e_j$ ($n\in\N$), which, in light of the increasing supports of the vectors, can not be transferred through a spreading model.\smallskip

In the second main result of this note we show that, indeed, the problem with vectors having `long support' can not be circumvented. The result, whose formal statement in given in Theorem B below, solves in the negative \cite[Problem 5.11]{HKR} in a very strong way.
\begin{theoremB} For every $\e>0$ there exists a Banach space $X$ every whose spreading model is isomorphic to $c_0$ and such that $K(X)\leq1+\e$.
\end{theoremB}

In conclusion to this section, let us mention that our notation in this note is standard and follows, \emph{e.g.}, \cite{ak}. Let us just remind here (see, \emph{e.g.}, \cite[p.~53]{ak}) that an unconditional basic sequence $(e_j)_{j=1}^\infty$ in a Banach space $X$ is \emph{suppression $1$-unconditional} if for every finite subset $A$ of $\N$ one has
$$\left\|\sum_{j\in A}\alpha_je_j\right\|\leq \left\|\sum_{j\in\N}\alpha_j e_j\right\|\qquad\left((\alpha_j)_{j=1}^\infty\in c_{00}\right).$$

\section{A symmetric version of the Elton--Odell theorem}\label{Sec: Ks and K}
\begin{proof}[Proof of Theorem A] We can assume that $X$ contains no copy of $\ell_1$, since otherwise $K^s(X)=2$, in light of James' non-distortion theorem \cite{james-distortion}. Therefore Rosenthal's $\ell_1$ theorem \cite{Rosenthal l1} yields the existence of a weakly null normalised sequence $(x_j)_{j=1}^\infty$ in $X$. Due to the Bessaga--Pe\l czy\'nski selection principle (see, \emph{e.g.}, \cite[Proposition 1.5.4]{ak}), we can additionally assume that $(x_j)_{j=1}^\infty$ is a basic sequence with basis constant less than $4/3$.

By James' non-distortion theorem, we may also assume that $X$ contains no copy of $c_0$. Consequently, the proof of the Elton--Odell theorem (see \cite[Remarks.~(1)]{E-O}) yields the existence of a normalised block sequence $(u_j)_{j=1}^\infty$ of $(x_j)_{j=1}^\infty$, which is a $(1+\e)$-separated sequence, for some $\e>0$.\smallskip

Let us now fix a parameter $\eta>0$ with $\eta<\min\{\e,1/2\}$ and consider the colouring $c$ of $[\N]^2$ given by: $$\{n,k\}\mapsto\begin{cases} (>) & \text{if }\; \|u_n +u_k\|>1+\eta \\ (\leq) & \text{if }\; \|u_n +u_k\|\leq1+\eta.\end{cases}$$
By Ramsey's theorem \cite{Ramsey}, we can select an infinite monochromatic subset $M$ of $\N$, \emph{i.e.}, such that $c$ is constant on $[M]^2$. In the case where the colour of every pair $\{n,k\}\in M$ is $(>)$, then $(u_j)_{j\in M}$ is evidently a symmetrically $(1+\eta)$-separated sequence, and we are done.

In the other case, up to passing to a subsequence, we can assume that $\|u_n+u_k\|\leq1+\eta$ for distinct $n,k\in\N$; therefore the vectors
$$y_j:=\frac{u_1+ u_{j+1}}{1+\eta}\qquad (j\in\N)$$
belong to the unit ball of $X$. Finally, for distinct $n,k\in\N$ we have
$$\|y_n-y_k\|=\frac{1}{1+\eta}\|u_{n+1}-u_{k+1}\|\geq\frac{1+\e}{1+\eta}$$
and (exploiting that the basis constant of $(u_j)_{j=1}^\infty$ is less than $4/3$)
$$\|y_n+y_k\|=\frac{1}{1+\eta}\|2u_1+u_{n+1}+u_{k+1}\|\geq\frac{3}{4(1+\eta)} \|2u_1\|=\frac{3}{2(1+\eta)}.$$
Therefore, the sequence $(y_j)_{j=1}^\infty$ is symmetrically $r$-separated, where $r:=\min\left\{\frac{1+\e}{1+\eta}, \frac{3}{2(1+\eta)}\right\}>1$, and we are done.
\end{proof}

It is perhaps clear that the above argument admits a quantitative counterpart, upon choosing the parameters in the optimal way. More precisely, we have the following proposition.
\begin{proposition}\label{Prop sqrt} Assume that the unit ball of a Banach space $X$ contains a weakly null $(1+\e)$-separated sequence, where $\e\in(0,1)$. Then, the unit ball also contains a symmetrically $\sqrt{1+\e}$-separated sequence.
\end{proposition}

\begin{proof} Choose one such weakly null $(1+\e)$-separated sequence $(u_j)_{j=1}^\infty$ and, up to passing to a subsequence, assume that it is basic, with basis constant less than $2/(1+\e)$ (here, we use that $\e<1$). Then, choose $\eta:=\sqrt{1+\e}-1$ and apply Ramsey's theorem as above. 

In the case where the infinite set $M$ has colour $(>)$, the sequence $(u_j)_{j=1}^\infty$ is already symmetrically $\sqrt{1+\e}$-separated; in the other case, the sequence $(y_j)_{j=1}^\infty$ as in the previous proof is immediately seen to be symmetrically $\sqrt{1+\e}$-separated.
\end{proof}

The above proposition allows us to obtain an interesting estimate for $K^s(X)$ in terms of $K(X)$ for certain classes of Banach spaces. In order to explain this, we need some results from \cite{DORR} and \cite{MaPa}. In the paper \cite{DORR}, Dronka, Olszowy and Rybarska-Rusinek introduced a constant, denoted $\gamma_0(X)$, as follows:
$$\gamma_0(X):=\sup\left\{r\geq0\colon \exists\, (x_n)_{n=1}^\infty\subseteq B_X \text{ weakly null and }r\text{-separated} \right\}.$$
With this notation, we can restate Proposition \ref{Prop sqrt} as the validity of the inequality, for every Banach space $X$,
\begin{equation}\label{Ks^2 geq gamma0}
K^s(X)\geq\sqrt{\gamma_0(X)}.
\end{equation}

In the same paper, the authors show that $\gamma_0(X)=K(X)$ for every reflexive Banach space $X$ with a \emph{co-monotone} Schauder basis $(e_i)_{i=1}^\infty$, namely such that
$$\left\|\sum_{i=n}^\infty a_i e_i\right\|\leq\left\|\sum_{i=1}^\infty a_i e_i\right\|,$$
for every $n\in\N$ and $(a_i)_{i=1}^\infty\in c_{00}$.\smallskip

This result has been generalised by Maluta and Papini \cite{MaPa} to the class of reflexive Banach spaces with the non-strict Opial property. Let us recall that a Banach space $X$ has the \emph{non-strict Opial property} if for every sequence $(x_n)_{n=1}^\infty$ in $X$ with weak limit $x$ and every $y\in X$
$$\liminf_{n\to\infty}\|x_n-x\|\leq\liminf_{n\to\infty}\|x_n-y\|.$$

We refer to \cite[Theorem 4.1]{MaPa} for the proof that $K(X)=\gamma_0(X)$, for reflexive Banach spaces $X$ with the non-strict Opial property and to \cite[Proposition 4.2]{MaPa} where it is proved that this is, indeed, a generalisation of \cite{DORR}. Let us also observe that every suppression $1$-unconditional basis is obviously co-monotone. Combining these results with (\ref{Ks^2 geq gamma0}), we therefore arrive at the following corollary.
\begin{corollary}\label{Cor: refl Ks>K} Let $X$ be a reflexive Banach space with the non-strict Opial property. Then
$$K^s(X)\geq\sqrt{K(X)}.$$
\end{corollary}

It is a classical result of James \cite{james-bases} that a Banach space with unconditional basis is reflexive, provided it contains no copies of $c_0$ or $\ell_1$. Since for Banach spaces containing a copy of $c_0$ or $\ell_1$ the conclusion of the above corollary is obviously true, we also obtain the following result.
\begin{corollary} Let $X$ be a Banach space with a suppression $1$-unconditional Schauder basis. Then
$$K^s(X)\geq\sqrt{K(X)}.$$
\end{corollary}

In conclusion to this section, let us notice that the two above corollaries leave plenty of space for further investigation on the comparison between $K^s(X)$ and $K(X)$, for a given infinite-dimensional Banach space. In particular, we don't know whether the estimate contained in the conclusion to Corollary \ref{Cor: refl Ks>K} is sharp. Moreover, it would be interesting to know how large the gap $K(X)-K^s(X)$ could be; for example, what is the minimal possible value for $K^s(X)$ for a Banach space $X$ such that $K(X)=2$?

\section{A second proof of Theorem A}\label{Sec: 2nd proof}
In this short part, we shall give a second proof of Theorem A, in the language of spreading models. Very recently, Freeman, Odell, Sari, and Zheng \cite{FOSZ} have distilled a sufficient condition for a Banach space to contain $c_0$, which readily implies the Elton--Odell theorem; we shall show below how to directly derive the symmetric version of the theorem from their result. Since such proof is actually just a variation of the argument above, we shall only sketch it. Let us start by stating the result from \cite{FOSZ} that we need; let us refer to \cite{BeaLa} or \cite{Odell-stability} for basic notions on spreading models.
\begin{theorem}[{\cite[Theorem 4.1]{FOSZ}}] Let $(x_j)_{j=1}^\infty$ be a normalised, weakly null basis for a Banach space $X$ that contains no copy of $\ell_1$. Assume that whenever $(y_j)_{j=1}^\infty$ is a normalised, weakly null block sequence of $(x_j)_{j=1}^\infty$ with spreading model $(e_j)_{j=1}^\infty$ one has $\|e_1-e_2\|=1$. Then $X$ contains a copy of $c_0$.
\end{theorem}

\begin{proof}[Second proof of Theorem A] By James' non-distortion theorem, we can assume that $X$ contains no copies of $c_0$ or $\ell_1$. Consequently, up to passing to a subspace, we may assume that $X$ admits a normalised, weakly null Schauder basis $(x_j)_{j=1}^\infty$. According to \cite[Theorem 4.1]{FOSZ}, we therefore derive the existence of a normalised, weakly null block basic sequence $(y_j)_{j=1}^\infty$ of $(x_j)_{j=1}^\infty$ that admits a spreading model $(e_j)_{j=1}^\infty$ satisfying $\|e_1-e_2\|>1$. Indeed, $(y_j)_{j=1}^\infty$ being weakly null, $(e_j)_{j=1}^\infty$ is suppression $1$-unconditional (see, \emph{e.g.}, \cite[Proposition I.5.1]{BeaLa}), whence $\|e_1-e_2\|\geq \|e_1\|=1$.\smallskip

We now distinguish two cases: if $\|e_1+e_2\|>1$, it readily follows from the definition of spreading models that there exists $k\in\N$ such that $(y_j)_{j=1}^\infty$ is symmetrically $(1+\e)$-separated, for some $\e>0$.

In the other case, where $\|e_1+e_2\|=1$, let $\e>0$ be such that $\|e_1-e_2\|\geq1+\e$ and consider the vectors $f_j:=e_1+e_{j+1}$ ($j\in\{1,2\}$). Evidently, $\|f_1\|=1$, $\|f_1-f_2\|=\|e_1-e_2\| \geq1+\e$ and $\|f_1+f_2\|=\|2e_1+e_2+e_3\|\geq2$, by the suppression $1$-unconditionality. Let us now fix a small parameter $\eta>0$; the above inequalities imply, up to discarding finitely many terms from the sequence $(y_j)_{j=1}^\infty$, that $\|y_1+y_{j+1}\|\leq1+\eta$, $\|(y_1+y_{j+1})-(y_1+y_{k+1})\|\geq1+\e-\eta$ and $\|(y_1+y_{j+1})+(y_1+y_{k+1})\|\geq2-\eta$, for distinct $j,k\in\N$. Therefore, the vectors
$$u_j:=\frac{y_1+y_{j+1}}{1+\eta}\qquad (j\in\N)$$
constitute the desired symmetrically separated sequence in $B_X$, with separation at least $\frac{1+\e-\eta}{1+\eta}>1$ (provided $\eta$ is chosen sufficiently small).
\end{proof}

\section{$c_0$ spreading models}\label{Sec: c0 spread}
The main goal of the present section is to show that the assumption on a Banach space $X$ to admit a spreading model isomorphic to $c_0$ is not sufficient to imply any estimate on its Kottman's constant. In particular, we will prove Theorem B and we will give one its generalisation, in the form of an upper estimate for the Kottman's constant of a class of \emph{asymptotically}-$c_0$ Banach spaces.\smallskip

Let us start by briefly reminding the definition and basic properties of Tsirelson's space, \cite{Tsirelson, FiJo} (also see, \emph{e.g.}, \cite{CaSh}). If $E$ and $F$ are finite subsets of $\N$ we write $E<F$ as a shorthand for $\max E< \min F$; in the case where $E=\{k\}$ is a singleton, we write $k<F$ in place of $\{k\}<F$. Analogous meaning is given to the expressions $E\leq F$, or $k\leq F$. Moreover, for a vector $x\in c_{00}$ and a finite subset $E$ of $\N$, we shall denote by $Ex$ the vector $(\chi_E(j)\cdot x(j))_{j=1}^\infty$.
\begin{definition}[\cite{FiJo}] Let $\theta\in(0,1)$ and denote by $\n_{T_\theta}$ the unique norm on $c_{00}$ that satisfies the following equation
$$\|x\|_{T_\theta}=\max\left\{\|x\|_\infty, \theta\sup \left\{\sum_{j=1}^k \|E_jx\|_{T_\theta} \colon k\in\N, k\leq E_1<\dots<E_k \right\} \right\}\qquad (x\in c_{00}).$$
\emph{Tsirelson's space} $T_\theta$ is the completion of $(c_{00},\n_{T_\theta})$.
\end{definition}

It follows easily from the definition that the canonical basis $(e_j)_{j=1}^\infty$ is a suppression $1$-unconditional basis for $T_\theta$. It is also clear that for every normalised block sequence $(u_j)_{j=1}^\infty$ of the canonical basis and for every $k\in\N$ there exists a finite subsequence $u_{j_1},\dots,u_{j_k}$ such that
$$\theta\cdot\sum_{i=1}^k|\alpha_i|\leq \left\|\sum_{i=1}^k \alpha_i u_{j_i}\right\|_{T_\theta} \leq\sum_{i=1}^k|\alpha_i|, \qquad(\alpha_1,\dots,\alpha_k\in\R).$$

We will be interested in the dual to $T_\theta$, nowadays known as the \emph{original Tsirelson's space} and denoted $T^*_\theta$. Such Banach space  was constructed by Tsirelson in \cite{Tsirelson} as the first example of an infinite-dimensional Banach space that contains no copy of $c_0$ or $\ell_p$ ($1\leq p<\infty$).

Standard duality arguments prove that the biorthogonal functionals $(e^*_j)_{j=1}^\infty$ to $(e_j)_{j=1}^\infty$ constitute a suppression $1$-unconditional Schauder basis for $T^*_\theta$. Moreover, for every normalised block sequence $(u_j)_{j=1}^\infty$ of the canonical basis and for every $k\in\N$ there exists a finite subsequence $u_{j_1},\dots,u_{j_k}$ such that
\begin{equation}\label{eq: c0 in T*}
\max_{i=1,\dots,k}|\alpha_i|\leq \left\|\sum_{i=1}^k \alpha_i u_{j_i}\right\|_{T^*_\theta} \leq\frac{1}{\theta} \max_{i=1,\dots,k}|\alpha_i|, \qquad(\alpha_1,\dots,\alpha_k\in\R).    
\end{equation}

An easy deduction from the above property and the Bessaga--Pe\l czy\'nski selection principle is the well known fact that every spreading model of $T^*_\theta$ is isomorphic to $c_0$ (see, \emph{e.g.}, \cite[p.~121]{BeaLa}). Equally well-know is that $T^*_\theta$ is a reflexive Banach space. For a discussion of all such properties, we refer to Chapter I in the aforementioned monograph \cite{CaSh}.\smallskip

We shall now compute the (symmetric) Kottman's constant of $T^*_\theta$.

\begin{theorem}\label{Ks for T*} Let $\theta\in(0,1)$ and $T^*_\theta$ be the original Tsirelson's space defined above. Then
$$K(T^*_\theta)=K^s(T^*_\theta)=\begin{cases}2 & \text{if }\theta\in(0,1/2]\\ 1/\theta & \text{if }\theta\in(1/2,1). \end{cases}$$
\end{theorem}
At the risk of stating the obvious, let us note that Theorem B is a direct consequence of Theorem \ref{Ks for T*}. We could derive the proof of the result under consideration from the already mentioned result by Maluta and Papini concerning reflexive Banach spaces with the non-strict Opial property; however, we prefer to offer a direct argument, based on the suppression $1$-unconditionality. As the reader will see, the proof is based on a sliding hump argument, somewhat similar to Kottman's proof (\cite{Kottman}) that $K(\ell_p)=2^{1/p}$, for $1\leq p<\infty$.

\begin{proof} Let us start with a lower bound for $K^s(T^*_\theta)$, which is simply obtained by checking the canonical basis $(e^*_j)_{j=2}^\infty$. Indeed, given natural numbers $2\leq i<j$, in the definition of $\n_{T_\theta}$ the optimal choice of sets $E_1,\dots,E_k$ is clearly given by $E_1=\{i\}$ and $E_2=\{j\}$. Consequently, we obtain
$$\|e_i\pm e_j\|_{T_\theta}=\max\left\{\|e_i\pm e_j\|_\infty,2\theta \right\}=\max\left\{1,2\theta \right\},$$
whence
$$2=\langle e^*_i\pm e^*_j,e_i\pm e_j\rangle\leq\max\left\{1,2\theta \right\} \cdot\|e^*_i\pm e^*_j\|_{T^*_\theta}.$$
This leads us to the desired lower bound
$$K^s(T^*_\theta)\geq\begin{cases}2 & \text{if }\theta\in(0,1/2]\\ 1/\theta & \text{if }\theta\in(1/2,1). \end{cases}$$\smallskip

Therefore, the argument will be concluded when we show that $K(T^*_\theta)\leq1/\theta$ whenever $\theta\in(1/2,1)$. Let us thus pick an arbitrary $r$-separated sequence $(x_j)_{j=1}^\infty$ in the unit ball of $T^*_\theta$ and assume, up to passing to a subsequence, that it admits a weak limit, say $x$. 

Let us now fix arbitrarily a parameter $\eta>0$; we may then find a finite subset $E$ of $\N$, of the form $E=\{1,\dots,N\}$, such that $\|x-Ex\|_{T^*_\theta}<\eta$. Since $(x_j)_{j=1}^\infty$ converges to $x$ weakly, up to discarding finitely many terms from the sequence, we can additionally assume that $\|Ex_i -Ex_j\|_{T^*_\theta}<\eta$, for every $i,j\in\N$.\smallskip

Setting $E^\complement:=\N\setminus E$, we therefore obtain, for distinct $i,j\in\N$,
\begin{equation}\label{eq1: proof T*}
r\leq\|x_i-x_j\|_{T^*_\theta}\leq \left\|E^\complement x_i- E^\complement x_j\right\|_{T^*_\theta} +\eta = \left\|E^\complement(x_i-x) - E^\complement(x_j-x)\right\|_{T^*_\theta} +\eta.
\end{equation}

Let us now observe that $\left(E^\complement(x_j-x)\right)_{j=1}^\infty$ is a weakly null sequence, whose terms satisfy
$$\|E^\complement(x_j-x)\|_{T^*_\theta}\leq\|E^\complement x_j\|_{T^*_\theta}+ \|x-Ex\|_{T^*_\theta}\leq 1+\eta\qquad(j\in\N),$$
where we used the suppression $1$-unconditionality in the last inequality. Consequently, up to passing to one more subsequence, we can assume that there exists a block sequence $(u_j)_{j=1}^\infty$ of the canonical basis such that $\|E^\complement(x_j-x)-u_j\|_{T^*_\theta} <\eta$. Evidently, $\|u_j\|_{T^*_\theta}<1+2\eta$. Up to passing to one more, last, subsequence, (\ref{eq: c0 in T*}) then yields
$$\|u_1-u_2\|_{T^*_\theta}\leqslant\frac{1}{\theta} \max\{\|u_1\|_{T^*_\theta}, \|u_2\|_{T^*_\theta}\}\leq \frac{1+2\eta}{\theta}.$$

Finally, insertion of this last inequality into (\ref{eq1: proof T*}) leads us to
$$r\leq \|u_1-u_2\|_{T^*_\theta} +3\eta \leq 3\eta + \frac{1+2\eta}{\theta},$$
whence the desired estimate $K(T^*_\theta)\leq1/\theta$ follows, upon letting $\eta\to 0^+$.
\end{proof}

Let us observe that the proof of the upper estimate in the previous argument didn't depend on any specific property of the norm of $T^*_\theta$; it only depended on reflexivity, the $c_0$-behaviour of basic sequences contained in (\ref{eq: c0 in T*}) and the suppression $1$-unconditionality. We already commented that, more generally, we could have used the non-strict Opial property, instead of unconditionality. Therefore, the above proof also offers us a more general result, whose formulation requires the notion of \emph{asymptotic $c_0$} Banach space (see, \emph{e.g.}, \cite{AGR}, or \cite{MMT} for a more general definition that does not require the existence of a Schauder basis).

\begin{definition}[{\cite[Definition III.4.1]{AGR}}] A Banach space with a Schauder basis $(e_j)_{j=1}^\infty$ is said to be  \emph{$\lambda$-asymptotic $c_0$} if for every normalised block basis $(x_j)_{j=1}^\infty$ and every $n\in\N$ there exists a finite subsequence $x_{j_1},\dots,x_{j_n}$ which is $\lambda$-equivalent to the canonical basis of $\ell_\infty^n$.
\end{definition}

We are now in position to state and prove a general result that subsumes the above result on the original Tsirelson's space.
\begin{theorem}\label{Th: asymptotic c0} Let $X$ be a reflexive, $\lambda$-asymptotic $c_0$ Banach space with the non-strict Opial property. Then
$$K(X)\leq\lambda.$$
\end{theorem}
\begin{proof} According to \cite[Theorem 4.1]{MaPa}, we know that $K(X)=\gamma_0(X)$. Fixed arbitrarily $\eta>0$, we can therefore select a weakly null $(K(X)-\eta)$-separated sequence $(x_j)_{j=1}^\infty$ in the unit ball of $X$. Up to passing to a subsequence, we can assume that there exists a block basis $(u_j)_{j=1}^\infty$ of the basis $(e_j)_{j=1}^\infty$ such that $\|x_j-u_j\|<\eta$. The assumption $X$ being $\lambda$-asymptotic $c_0$ then implies, up to passing to a further subsequence,
$$\|u_j-u_i\|\leq\lambda \max\{\|u_j\|,\|u_i\|\}\leq(1+\eta)\lambda.$$
Consequently,
$$K(X)-\eta\leq \|x_i-x_j\|\leq \|u_j-u_i\|+2\eta\leq (1+\eta)\lambda+ 2\eta,$$
whence the conclusion follows by letting $\eta\to 0^+$.
\end{proof}

In conclusion of our note, let us observe that all the assumptions in the above result are necessary. The Banach space $c_0$ itself is an obvious example that the assumption of reflexivity can not be dispensed with. More interesting is the fact that some `monotonicity' assumption is indeed necessary. This is consequence of the result by Castillo, Gonz\'ales, and Papini \cite[Theorem 4.2]{CGP} that every Banach space is isometric to a hyperplane of a Banach space with Kottman's constant equal to 2. When combined with the obvious fact that a Banach space is $\lambda$-asymptotic $c_0$ whenever some its hyperplane is so, we readily conclude that, for every $\lambda>1$ there exists a reflexive, $\lambda$-asymptotic $c_0$ Banach space whose Kottman's constant equals $2$.

\medskip{}
\textbf{Acknowledgements.} Part of the results in this paper originates from a conversation with Pavlos Motakis, at the conference \emph{Non Linear Functional Analysis} held at CIRM, Marseille. In particular, Pavlos Motakis suggested us the second proof of Theorem A presented in \S\ref{Sec: 2nd proof} and conjectured that some form of Theorem \ref{Th: asymptotic c0} might be true. The author is most grateful to Pavlos for sharing his insight with us and for his interest in the topic.



\begin{thebibliography}{99}
\bibitem{ak} F.~Albiac and N.~Kalton, \emph{Topics in Banach space theory}, Graduate Texts in Mathematics, \textbf{233}. Springer, New York, 2006.
\bibitem{AGR} S.A.~Argyros, G.~Godefroy, and H.P.~Rosenthal, Descriptive set theory and Banach spaces, \emph{Handbook of the geometry of Banach spaces}, Vol.~2, 1007--1069. North-Holland, Amsterdam, 2003.
\bibitem{BeaLa} B.~Beauzamy and J.T.~Laprest\'e, \emph{Mod\`eles \'etal\'es des espaces de Banach},  Travaux en Cours \textbf{4}, Hermann, Paris, 1984.
\bibitem{CaSh} P.G.~Casazza and T.J.~Shura, \emph{Tsirelson's space}, Lecture Notes in Mathematics, \textbf{1363}. Springer-Verlag, Berlin, 1989.
\bibitem{CGP} J.M.F.~Castillo, M.~Gonz\'{a}lez, and P.L.~Papini, New results on Kottman's constant, \emph{Banach J. Math.~Anal.} \textbf{11} (2017), 348--362.
\bibitem{CaPa} J.M.F.~Castillo and P.L.~Papini, On Kottman's constant in Banach spaces, \emph{Function Spaces IX}, \emph{Banach Center Publ.} \textbf{92} (2011), 75--84.
\bibitem{delpech} S.~Delpech, Separated sequences in asymptotically uniformly convex Banach spaces, \emph{Colloq.~Math.} \textbf{119} (2010), 123--125.
\bibitem{DORR} J.~Dronka, L.~Olszowy and L.~Rybarska-Rusinek, Separability of weakly convergent sequences in Banach spaces, \emph{Panamer. Math. J.} \textbf{16} (2006), 67--82.
\bibitem{E-O} J.~Elton and E.~Odell, The unit ball of every infinite-dimensional normed linear space contains a~$(1+\e)$-separated sequence, \emph{Colloq. Math.} \textbf{44} (1981), 105--109.
\bibitem{FiJo} T.~Figiel and W.~B.~Johnson, A uniformly convex Banach space which contains no $\ell_p$, \emph{Compositio Math.} \textbf{29} (1974), 179--190.
\bibitem{FOSS14} D.~Freeman, E.~Odell, B.~Sari, and Th.~Schlumprecht, Equilateral sets in uniformly smooth Banach spaces, \emph{Mathematika} \textbf{60} (2014), 219--231.
\bibitem{FOSZ} D.~Freeman, E.~Odell, B.~Sari, and B.~Zheng, On spreading sequences and asymptotic structures, \emph{Trans.~Amer.~Math.~Soc.} \textbf{370} (2018), 6933--6953.
\bibitem{GlMe} E.~Glakousakis and S.K.~Mercourakis, Antipodal sets in infinite dimensional Banach spaces, \texttt{arXiv:1801.02002}.
\bibitem{HKR} P.~H\'ajek, T.~Kania, and T.~Russo, Symmetrically separated sequences in the unit sphere of a Banach space, \textit{J. Funct. Anal.} \textbf{275} (2018), 3148--3168.
\bibitem{james-distortion} R.C.~James, Uniformly non-square Banach spaces, \emph{Ann. of Math.} \textbf{80} (1964), 542--550.
\bibitem{james-bases} R.C.~James, Bases and reflexivity of Banach spaces, \emph{Ann. of Math.} \textbf{52} (1950), 518--527.
\bibitem{Koszmider} P.~Koszmider, Uncountable equilateral sets in Banach spaces of the form $C(K)$, \emph{Israel J.~Math.} \textbf{224} (2018), 83--103.
\bibitem{Kottman} C.A.~Kottman, Subsets of the unit ball that are separated by more than one, \emph{Studia Math.} \textbf{53} (1975), 15--27.
\bibitem{KrP} A.~Kryczka and S.~Prus, Separated sequences in nonreflexive Banach spaces, \emph{Proc. Amer. Math. Soc.} \textbf{129} (2000), 155--163.
\bibitem{MaPa} E.~Maluta and P.L.~Papini, Estimates for Kottman's separation constant in reflexive Banach spaces, \emph{Colloq.~Math.} \textbf{117} (2009), 105--119.
\bibitem{MMT} B.~Maurey, V.D.~Milman, and N.~Tomczak-Jaegermann, Asymptotic infinite-dimensional theory of Banach spaces, \emph{Geometric aspects of functional analysis (Israel, 1992--1994)}, 149--175. 
Oper.~Theory Adv.~Appl. \textbf{77}, Birkh{\"a}user, Basel, 1995.
\bibitem{MV equilateral} S.K.~Mercourakis and G.~Vassiliadis, Equilateral sets in infinite dimensional Banach spaces, \emph{Proc. Amer. Math. Soc.} \textbf{142} (2014), 205--212.
\bibitem{mv} S.K.~Mercourakis and G.~Vassiliadis, Equilateral Sets in Banach Spaces of the form $C(K)$, \emph{Studia Math.} \textbf{231} (2015), 241--255.
\bibitem{NaSa} S.V.R.~Naidu and K.P.R.~Sastry, Convexity conditions in normed linear spaces, \emph{J. Reine Angew. Math.} \textbf{297} (1978), 35--53.
\bibitem{Odell-stability} E.~Odell, Stability in Banach spaces, \emph{Extracta Math.} \textbf{17} (2002), 385--425.
\bibitem{prus} S.~Prus, Constructing separated sequences in Banach spaces, \emph{Proc. Amer. Math. Soc.} \textbf{138} (2010), 225--234.
\bibitem{Ramsey} F.P.~Ramsey, On a Problem of Formal Logic, \emph{Proc. London Math. Soc.} \textbf{30} (1929), 264--286.
\bibitem{Riesz} F.~Riesz, \"Uber lineare Funktionalgleichungen, \emph{Acta Math.}, \textbf{41} (1916), 71--98.
\bibitem{Rosenthal l1} H.P.~Rosenthal, A characterization of Banach spaces containing $\ell_1$, \emph{Proc. Nat. Acad. Sci. U.S.A.} \textbf{71} (1974), 2411--2413.
\bibitem{Terenzi1} P.~Terenzi, Regular sequences in Banach spaces, \emph{Rend. Sem. Mat. Fis. Milano 57 (1987)}, 275--285.
\bibitem{Terenzi2} P.~Terenzi, Equilater sets in Banach spaces, \emph{Boll. Un. Mat. Ital. A (7)} \textbf{3} (1989), 119--124.
\bibitem{Tsirelson} B.S.~Tsirelson, Not every Banach space contains an embedding of $\ell_p$ or $c_0$, \emph{Funct. Anal. Appl.} \textbf{8} (1974), 138--141.
\end{thebibliography}
\end{document}